\documentclass[10pt,notitlepage,twoside,a4paper]{amsart}

\usepackage{amsfonts}
\usepackage{amsmath,amssymb,enumerate}

\usepackage{epsfig,fancyhdr,color}%,showkeys,amsmidx

\usepackage{psfrag}
\usepackage{amssymb}
\usepackage{amsmath,amsthm}
\usepackage{latexsym}
\usepackage{amscd}
\usepackage{psfrag}
\usepackage{graphicx}
\usepackage[latin1]{inputenc}
\usepackage[all]{xy}
\usepackage[mathcal]{eucal}
\usepackage{amsthm}

\definecolor{NoteColor}{rgb}{1,0,0}

\renewcommand{\textsc}{\textcolor{red}}

\newtheorem{theorem}{\rm\bf Theorem}[section]

\newenvironment{theoremintro}[2][Theorem]{\begin{trivlist}
\item[\hskip \labelsep {\bfseries #1}\hskip \labelsep {\bfseries #2}]}{\end{trivlist}}

\newtheorem{proposition}[theorem]{\rm\bf Proposition}
\newtheorem{lemma}[theorem]{\rm\bf Lemma}
\newtheorem{corollary}[theorem]{\rm\bf Corollary}

\theoremstyle{definition}

\theoremstyle{remark}
\newtheorem{remark}[theorem]{\rm\bf Remark}

\newtheorem{question}[theorem]{\rm\bf Question}

\def\interieur#1{\mathord{\mathop{\kern 0pt #1}\limits^\circ}}

\title[Teichm\"uller space]{On the inclusion of the quasiconformal Teichm\"uller space into the
length-spectrum Teichm\"uller space}

\author{D. Alessandrini}
\address{Daniele Alessandrini, D\'epartement de Math\'ematiques,
Universit\'e de Fribourg, Chemin du Mus\'ee 23, 1700 Fribourg,
Switzerland}
\email{daniele.alessandrini@gmail.com}

\author{L. Liu}
\address{Lixin Liu, Department of Mathematics, Sun Yat-Sen University, 510275, Guangzhou, P. R. China}
\email{mcsllx@mail.sysu.edu.cn}

\author{A. Papadopoulos}
\address{Athanase Papadopoulos,  Institut de Recherche Math\'ematique Avanc\'ee (Universit{\'e} de Strasbourg and CNRS),
7 rue Ren\'e Descartes,
 67084 Strasbourg Cedex, France} \email{athanase.papadopoulos@math.unistra.fr}

\author{W. Su}
\address{Weixu Su, Department of Mathematics, Fudan University, 200433, Shanghai, P. R.
China, and  Institut de Recherche Math\'ematique Avanc\'ee (Universit\'e de Strasbourg and CNRS), 7 rue Ren\'e Descartes, 67084 Strasbourg
Cedex, France}
\email{suweixu@gmail.com}

\date{\today}
\begin{document}

\begin{abstract}
Given a surface of infinite topological type, there are several Teichm\"uller spaces associated with it, depending on the basepoint and on the point of view that one uses to compare different complex structures. This paper is about the comparison between the  quasiconformal Teichm\"uller space and the length-spectrum Teichm\"uller space. We work under this hypothesis that the basepoint is upper-bounded and admits short interior curves. There is a natural inclusion of the quasiconformal space in the length-spectrum space. We prove that, under the above hypothesis, the image of this inclusion is nowhere dense in the length-spectrum space. As a corollary we find an explicit description of the length-spectrum Teichm\"uller space in terms of Fenchel-Nielsen coordinates and we prove that the length-spectrum Teichm\"uller space is path-connected.   
\end{abstract}

\maketitle

\bigskip

\medskip

\medskip

%\tableofcontents

\section{Introduction}\label{intro}

In this paper, the word \emph{surface} refers to a connected orientable surface of finite or infinite topological type. To motivate our results, we review some basic preliminary facts about Teichm\"uller spaces of surfaces of infinite topological type.

\subsection{The quasiconformal Teichm\"uller space}
Given a surface $S$, its Teichm\"uller space is a parameter space of some homotopy classes of complex structures on $S$. There are several possible ways for defining what is the set of homotopy classes one wants to parametrize and which topology one wants to put on this set. Usually it is not necessary to worry too much about the details, because in the most common case considered, i.e. the case when $S$ is a closed surface, the set is just the set of all possible homotopy classes of complex structures, and all ``reasonable'' possible definitions of a topology on that set give the same topology. Therefore, in the case of closed surfaces, this freedom to choose between several possible definitions is not a problem, it is, instead, a very useful tool in the theory: one can choose the definition that best suits the problem studied.

As soon as we leave the setting of closed surfaces, it is necessary to be more careful with the definitions. In this paper we deal with surfaces of infinite topological type, and in this case the different possible definitions do not always agree. First of all, it is necessary to choose a basepoint, i.e. a base complex structure $R$ on the surface $S$, and then to consider only the set of homotopy classes of complex structures on $S$ that are ``comparable'' with $R$ in a suitable sense. This notion of comparability usually suggests a good definition of the topology. For example, the most commonly used definition is what we call the \emph{quasiconformal} Teichm\"uller space, where the set $\mathcal{T}_{qc}(R)$ parametrizes all the homotopy classes of complex structures $X$ on $S$ that are quasiconformally equivalent to $R$, i.e. such that there exists a quasiconformal homeomorphism between $R$ and $X$ that is homotopic to the identity of $S$. (Note that the space $\mathcal{T}_{qc}(R)$ we consider here is the reduced Teichm\"uller space i.e. homotopies need not fix the ideal boundary point-wise.)

The topology on $\mathcal{T}_{qc}(R)$ is given by the quasiconformal distance $d_{qc}$, also called the Teichm\"uller distance, defined using quasiconformal dilatations of quasiconformal homeomorphisms: for any two homotopy classes of complex structures $X, Y\in \mathcal{T}_{qc}(R)$, their quasiconformal distance $d_{qc}(X,Y)$ is defined as
$$d_{qc}(X,Y)=\frac{1}{2} \log  \inf_{f}K(f)$$
where $K(f)$ is the quasiconformal dilatation of a quasiconformal homeomorphism
$f:X\to Y$ which is homotopic to the identity.

\subsection{Fenchel-Nielsen coordinates}

In \cite{ALPSS}, we studied the quasiconformal Teichm\"uller space of a surface of infinite topological type using pair of pants decompositions and Fenchel-Nielsen coordinates. We will use this technique also in this paper, so we recall some of the main facts we need here. The definition of these coordinates depends on the interpretation of every complex structure as a hyperbolic metric on the surface. To do so, we use the intrinsic hyperbolic metric on a complex surface, defined by Bers. For complex structures of the first type (i.e. if the ideal boundary is empty) this metric is just the Poincar\'e metric, but for complex structures of the second type it is different from this well known metric. For these notions and for other related notions, we refer the reader to the paper \cite{ALPSS} for the details of this definition and for explanations about how the intrinsic metric may differ from the Poincar\'e metric. This metric has the property that every puncture of the surface shows one of the following behaviors:
\begin{enumerate}
\item It has a neighborhood isometric to a cusp, i.e., the
quotient of $\{z = x + iy \in \mathbb{H}^2 \ | \ a < y\}$, for some $a > 0$, by the group
generated by the translation $z \mapsto z + 1$.
\item It is possible to glue to the puncture a boundary component that is a simple closed geodesic for the hyperbolic metric. Punctures of this kind will be called boundary components, and the boundary geodesic will be considered as part of the surface.
\end{enumerate}

For infinite-type surfaces we proved in \cite{ALPSS} that given a topological pair of pants decomposition $\mathcal{P} = \{C_i\}$ of $S$ and a complex structure $X$ on $S$, it is always possible to find geodesics $\{\gamma_i\}$ for the intrinsic metric of $X$ such that each $\gamma_i$ is homotopic to $C_i$ and the set $\{\gamma_i\}$ is again a pair of pants decomposition of $S$. To every curve $C_i$ of the topological decomposition we can associate two numbers $(\ell_X(C_i), \tau_X(C_i))$, where $\ell_X(C_i)$ is the length in $X$ of the geodesic $\gamma_i$ and $\tau_X(C_i)$ is the twist parameter between the two pairs of pants (which can be the same) with geodesic boundary adjacent to $\gamma_i$. In this paper, the twist parameter is a length, and this is a slight change in notation with reference to the paper \cite{ALPSS}, where it was an ``angle'' parameter.

Let $R$ be a complex structure equipped with a geodesic pants decomposition $\mathcal{P}=\{C_i\}$. We say that the pair $(R,\mathcal{P} )$ is \emph{upper-bounded}  if $\sup_{C_i} \ell_R(C_i)< \infty$. We say that the pair $(R,\mathcal{P} )$ is \emph{lower-bounded}  if $\inf_{C_i} \ell_R(C_i) > 0$. If $(R,\mathcal{P} )$ is both upper-bounded and lower-bounded, then we say that $(R,\mathcal{P} )$ satisfies \emph{Shiga's property} (this property was first used by Shiga in \cite{Shiga}). Note that if $R$ is of finite type, then any pants decomposition of $R$ satisfies Shiga's property. Shiga's property was introduced in \cite{Shiga}, and we used it in our papers \cite{ALPS2} and \cite{ALPS1}. We note however that this property is used in a weaker form in these  papers \cite{ALPS2} and \cite{ALPS1}, and this is also the form which will be useful in the present paper. In fact, we shall say from now on that Shiga's property holds for the pair $(R,\mathcal{P} )$ if the pants decomposition $\mathcal{P} $ is upper-bounded and there exists a positive constant $\delta$ such that $ \ell_R(C_i) > \delta$ for any $C_i\in \mathcal{P}$ which is in the interior of the surface.

In this paper, we will often use another condition, we say that $(R,\mathcal{P})$ admits \emph{short interior curves} if there is a sequence of curves of the pair of pants decomposition $\alpha_k = C_{i_k}$ such that the curves $\alpha_k$ are not boundary components of the surface and such that $\ell_R(\alpha_i)$ tends to zero.  

In the paper \cite{ALPSS}, we proved that if $(R,\mathcal{P} )$ is upper-bounded, then the quasiconformal Teichm\"uller space $(\mathcal{T}_{qc}(R), d_{qc})$ is locally bi-Lipschitz equivalent to the sequence space $\ell^\infty$, using Fenchel-Nielsen coordinates. An analogous result, in the case of the non-reduced Teichm\"uller space, is due to Fletcher, see \cite{Fletcher} and the survey by Fletcher and Markovic \cite{FMH}.

\subsection{The length-spectrum Teichm\"uller space}

In this paper we study a different definition of Teichm\"uller space, which we call the length-spectrum Teichm\"uller space. The definition of this space and of its distance depend on a measure of how the lengths of essential curves change when we change the complex structure.

A simple closed curve on a surface is said to be \emph{essential} if it is not homotopic to a point or to a puncture (but it can be homotopic to a boundary component). We denote by $\mathcal{S} $ the set of homotopy classes of essential simple closed curves on $R$.

Given a complex structure $X$ on $S$ and an essential simple closed curve $\gamma$ on $S$, we denote by $\ell_X(\gamma)$ the length, for the intrinsic metric on $X$, of the unique geodesic that is homotopic to $\gamma$. The value $\ell_X(\gamma)$ does not change if we take another complex structure homotopic to $X$, hence this function is well defined on homotopy classes of complex structures on $S$.

Given two homotopy classes $X, Y$ of complex structures on $S$, we define the functional

$$L(X, Y)=\sup_{\gamma\in \mathcal{S} } \left\{\frac{\ell_X(\gamma)}{\ell_Y(\gamma)}, \frac{\ell_Y(\gamma)}{\ell_X(\gamma)}\right\} \leq \infty.$$

Given a base complex structure $R$ on $S$, the \emph{length-spectrum} Teichm\"uller space $\mathcal{T}_{ls}(R)$ is the space of homotopy classes of complex structures $X$ on $S$ satisfying $L(R, X) < \infty$.

For any two distinct elements $X,Y \in \mathcal{T}_{ls}(R)$, we have $1<L(X, Y) < \infty$.
We define a metric $d_{ls}$ on $\mathcal{T}_{ls}(R)$, called the \emph{length-spectrum distance}, by setting
$$d_{ls}(X,Y)=\frac{1}{2}\log L(X,Y).$$

Historically, the length-spectrum distance was defined before the length-spectrum Teichm\"uller space. Initially people considered this distance as a distance on the quasiconformal Teichm\"uller space: they studied the metric space $(\mathcal{T}_{qc}(R), d_{ls})$. For finite type surfaces this is a perfectly fine distance on $\mathcal{T}_{qc}(R)$, it makes this space complete and it induces the ordinary topology, and it is more suitable for studying problems regarding lengths of geodesics. For an example of a paper studying this space, see \cite{CR}, some of whose results we will use in the following.

The first paper dealing with the space $(\mathcal{T}_{qc}(R), d_{ls})$ in the case of surfaces of infinite type was \cite{Shiga}. 

We proved in \cite{ALPS1} that the metric space $(\mathcal{T}_{qc}(R), d_{ls})$ is, in general, not complete. More precisely this happens if there exists a pair of pants decomposition $\mathcal{P} = \{C_i\}$ such that the pair $(R, \mathcal{P})$ admits short interior curves, i.e. if there is a sequence of curves of the pair of pants decomposition $\alpha_k = C_{i_k}$ contained in the interior of $R$ with $\ell_R(\alpha_i) \to 0$.
The idea was to construct a sequence of hyperbolic metrics by large twists along short curves. To be more precise, for any $X \in \mathcal{T}_{qc}(R)$, we have again $\ell_X(\alpha_i)= \epsilon_i \to 0$. Denote by $\tau^{t}_{\alpha}(X)$ the surface obtained from $X$ by a twist of magnitude $t$ along $\alpha$ and let $X_i=\tau^{t_i}_{\alpha_i}(X)$, with $t_i=\log |\log \epsilon_i|$. Then we proved that $d_{qc}(X, X_i)\to \infty$, while $d_{ls}(X,X_i)\to 0$. If we set $Y_n=\tau^{t_n}_{\alpha_n} \circ \cdots\circ \tau^{t_2}_{\alpha_2}  \circ\tau^{t_1}_{\alpha_1}(X)$ and if we define $Y_\infty$ to be the surface obtained from $X$ by a twist of magnitude $t_i$ along $\alpha_i$ for every $i$, then a similar argument shows that $Y_\infty \in \mathcal{T}_{ls}(R) \setminus \mathcal{T}_{qc}(R)$ and $\lim_{n \to\infty}d_{ls}(Y_n, Y_\infty)=0$; see \cite{ALPS1}, \cite{ALPS2} for more details. We shall give another proof of the last result in Section \ref{sec:main} below.

The fact that the metric space $(\mathcal{T}_{qc}(R), d_{ls})$ is, in general, not complete is an indication of the fact that $\mathcal{T}_{qc}(R)$ is not the right space for this distance. The length-spectrum Teichm\"uller space was defined in \cite{LP}, with the idea that it was the most natural space for that distance, and it was studied in \cite{ALPS1}, \cite{ALPS2}. We proved in \cite{ALPS1} that for every base complex structure $R$, the metric space $(\mathcal{T}_{ls}(R), d_{ls})$ is complete. This result answered a question raised in \cite{LP} (Question 2.22).

Other properties, such as connectedness and contractibility, are unknown in the general case for surfaces of infinite type. If the basepoint $R$ satisfies Shiga's condition, then it follows from the main result of \cite{ALPS2} that $(\mathcal{T}_{ls}(R), d_{ls})$ is homeomorphic to the sequence space $\ell^\infty$ with an homeomorphism that is locally bi-Lipschitz, and, in particular, the space is contractible. One of the results we prove in the present paper is the following.

\begin{theoremintro}{\ref{thm:path}}
If $(R,\mathcal{P} )$ is upper-bounded and it admits short interior curves, then
$(\mathcal{T}_{ls}(R),d_{ls} )$ is path-connected.
\end{theoremintro}

To obtain this result, we will need some results about the comparison between the quasiconformal and the length-spectrum spaces (see below), and we will also need the following explicit characterization of the length-spectrum Teichm\"uller space in terms of Fenchel-Nielsen coordinates, which is interesting in itself: 

\begin{theoremintro}{\ref{thm:iff}}
Assume $X=(\ell_X(C_i), \tau_X(C_i))$. $X$ lies in $\mathcal{T}_{ls}(R)$ if and only if there is a constant $N>0$ such that for each $i$,
$$ \left|\log \frac{\ell_X(C_i)}{\ell_Y(C_i)}\right| < N $$
and
$$|\tau_X(C_i)-\tau_R(C_i)|< N \max \{|\log\ell_R(C_i)|, 1\}.$$
\end{theoremintro}

\subsection{Comparison between the two spaces}

It is interesting to compare the two spaces $(\mathcal{T}_{ls}(R),d_{ls} )$ with $(\mathcal{T}_{qc}(R),d_{qc} )$.

%It is known that $(\mathcal{T}_{qc}(R),d_{qc} )$ is complete, path-connected and contractible.

A classical result of Sorvali \cite{Sor} and Wolpert \cite{Wolpert79} states that for any $K$-quasi\-con\-for\-mal map $f: X \to Y$ and any $\gamma\in \mathcal{S}$, we have
$$\frac{1}{K}\leq \frac{\ell_Y(f(\gamma))}{\ell_X(\gamma)}\leq K.$$
It follows from this result that there is a natural inclusion map
\[
I: (\mathcal{T}_{qc}(R),d_{qc} ) \to (\mathcal{T}_{ls}(R),d_{ls})
\]
and that this map is $1$-Lipschitz.

In \cite{LP} we proved that if $R$ satisfies the Shiga's condition, then this inclusion is surjective, showing that under this hypothesis we have $\mathcal{T}_{ls}(R) = \mathcal{T}_{qc}(R)$ as sets. In the same paper we also gave an example of a complex structure $R$ such that the inclusion map $I$ is not surjective.

The inverse map of $I$ (defined on the image set $\mathcal{T}_{qc}(R)$) is not always continuous.  Shiga gave in \cite{Shiga} an example of a hyperbolic structure $R$ on a surface of infinite type and a
sequence $(R_n)$ of hyperbolic structures in
 $\mathcal{T}_{ls}(R)\cap \mathcal{T}_{qc}(R)$ which satisfy
$$d_{ls}(R_n,R)\to 0, \ \mathrm{while} \ d_{qc}(R_n,R)\to \infty.$$
In particular, the metrics $d_{ls}$ and $d_{qc}$ do not induce the same topology on
$\mathcal{T}_{qc}(R)$. A more general class of surfaces with the same behavior was described in the paper \cite{LSW} by Liu,  Sun and Wei.

In the same paper, Shiga showed that if the hyperbolic metric $R$ carries a geodesic pants decomposition that satisfies Shiga's condition, then
$d_{ls}$ and $d_{qc}$ induce the same topology on $\mathcal{T}_{qc}(R)$. In the paper \cite{ALPS2} we strengthened this result by showing that under Shiga's condition the inclusion map is locally bi-Lipschitz.

Several natural questions arise after this,  for instance :

\begin{enumerate}
\item Give necessary and sufficient conditions under which the inclusion map $I$ is surjective.
\item Under what conditions is the inverse map (defined on the image set) continuous ? Under what conditions is it  Lipschitz ? bi-Lipschitz ?
\item Are the two spaces  $(\mathcal{T}_{qc}(R),d_{qc} )$ and $(\mathcal{T}_{ls}(R),d_{ls})$  in the general case locally isometric to the infinite sequence space $\ell^\infty$ ? If not, are there other ``model spaces" to which such a space is locally isometric ?
\item \label{4} How does the image $I\left((\mathcal{T}_{qc}(R),d_{qc} )\right)$ sit in the space $(\mathcal{T}_{ls}(R),d_{ls})$ ? Is it dense ? Is it nowhere dense ?
\end{enumerate}

Some of these problems have been solved in the present paper.

Consider $\mathcal{T}_{qc}(R)$ as a subset in $\mathcal{T}_{ls}(R)$. A natural question which is asked in  \cite{ALPS1} is whether the subset $\mathcal{T}_{qc}(R)$ is dense in $(\mathcal{T}_{ls}(R),d_{ls} )$. This would tell us that $\mathcal{T}_{ls}(R)$ is the metric completion of $\mathcal{T}_{qc}(R)$ with reference to the distance $d_{ls}$, and it would also imply that $(\mathcal{T}_{ls}(R),d_{ls} )$ is connected (since the closure of a connected subset is also connected). In this paper, we will give a negative answer to this question. We prove the following:

\begin{theoremintro}{\ref{th:W}}
If $R$ admits a geodesic pants decomposition which is upper-bounded and admits short interior geodesics, then the space $(\mathcal{T}_{qc}(R),d_{ls} )$ is nowhere dense in $(\mathcal{T}_{ls}(R),d_{ls} )$.
\end{theoremintro}

The proof of Theorem \ref{th:W}  involves some estimates between quasiconformal dilatation and hyperbolic length under the twist deformation. Some of the techniques used in this paper are developed in our papers \cite{ALPSS} and \cite{ALPS2}. The upper-boundedness assumption of $(R,\mathcal{P} )$ is used here so that we can get a lower bound estimate of the length-spectrum distance under a twist. Without this assumption, the density of $\mathcal{T}_{qc}(R)$ in $\mathcal{T}_{ls}(R)$ is unknown.

It would be interesting to study the closure of $\mathcal{T}_{qc}(R)$ in $\mathcal{T}_{ls}(R)$. This space is the completion of $\mathcal{T}_{qc}(R)$ with reference to the length-spectrum metric.

Under the geometric conditions of Theorem \ref{th:W}, we showed in \cite{ALPS1} (Example 5.1) that the inverse of the inclusion map $I$ restricted to $\mathcal{T}_{qc}(R)$:  $(\mathcal{T}_{qc}(R), d_{ls}) \to (\mathcal{T}_{qc}(R), d_{qc})$ is nowhere continuous. We give another proof of this fact in the present paper (Proposition \ref{prop:inv}).

Regarding Item (\ref{4}) above, we prove (Proposition \ref{pro:boundary} below) that if the surface $R$ admits an upper-bounded pants decomposition $\mathcal{P}$ such that $(R,\mathcal{P})$ is upper-bounded and admits short interior curves, then there exists  a point in $\mathcal{T}_{ls}(R)\setminus \mathcal{T}_{qc}(R)$ which can be approximated by a sequence in $\mathcal{T}_{qc}(R)$ with the length-spectrum metric. This gives in particular a new proof of the fact (obtained in \cite{ALPS1}) that the space $\mathcal{T}_{qc}(R)$ equipped with the restriction of the metric $d_{ls}$ is not complete.

\section{Preliminaries}

Let $R$ be a base topological surface equipped with a hyperbolic structure $X$ and with a geodesic pants decomposition $\mathcal{P}=\{C_i\}$. The pieces of the decomposition (completions of connected components of the complements of the curves $C_i$) are spheres with three holes equipped with hyperbolic metrics, where a hole is either a cusp or a geodesic boundary component. We call such a piece a \emph{generalized} pair of pants to stress on the fact that it is not necessarily a hyperbolic pair of pants with three geodesic boundary componentss. To each $C_i\in \mathcal{P}$, we consider its length parameter $\ell_X(C_i)$ and its twist parameter $\tau_X(C_i)$. Recall that the latter is only
defined if $C_i$ is not a boundary component of $R$, and it is a measure of the relative twist amount along the geodesic  $C_i$ between the two generalized pairs of pants (which may be the same) that have this geodesic in common. The twist amount per unit time along $C_i$ is chosen so that a complete positive Dehn twist along $C_i$ changes the
twist parameters on $C_i$ by addition of $\ell_X(C_i)$.
For any hyperbolic metric $X$, its Fenchel-Nielsen parameters relative
to $\mathcal{P}$ is the collection of pairs
$$\{(\ell_X(C_i), \tau_X(C_i))\}_{i=1,2,\cdots}$$
where it is understood that if $C_i$ is a boundary component of $R$, then there
is no twist parameter associated to it, and instead of a pair $(\ell_X(C_i),\tau_X(C_i))$ we
have a single parameter $\ell_X(C_i)$.

There is an injective mapping from $\mathcal{T}_{qc}(R)$ or $\mathcal{T}_{ls}(R)$
to an infinite-dimensional real parameter space:
\begin{equation*}\label{equ:FN}
X\mapsto \left(\left(\log \frac{\ell_X(C_i)}{\ell_{R}(C_i)},\tau_X(C_i)-\tau_R(C_i)\right)\right)_{i=1,2, \cdots}.
\end{equation*}
If the image of $X$ belongs to $\ell^\infty$, then we say that $X$ is \emph{Fenchel-Nielsen bounded} (with respect to $(R,\mathcal{P}$)).

For each $C_i$ in the interior of $X$, there is a simple closed curve $\beta_i$ satisfying the following (see Figure \ref{fig:dual}):
\begin{enumerate}
\item $\beta_i$ and $C_i$ intersect minimally, that is, $i(C_i,\beta_i)=1$ or $2$;
\item $\beta_i$ does not intersect any $C_j, j \neq i$.

\end{enumerate}

\bigskip
  \begin{figure}[!hbp]
\centering
\includegraphics[width=.40\linewidth]{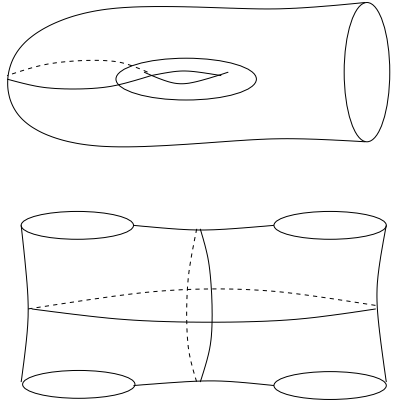}
\caption{\small {In each case, we have represented the curve $C_i$ and its dual curve $\beta_i$.}}
\label{fig:dual}
\end{figure}
\bigskip

 The following result is proved in  \cite{ALPSS}.
 \begin{lemma}\label{dual}
Suppose that the pants decomposition $\{C_i\}$ is upper-bounded by $M$, that is,  $\sup_i \{\ell_X(C_i)\}\leq M$. Then there exists a positive constant $\rho$ depending only on $M$ such that for each $i$, $\beta_i$ can be chosen so that  the intersection angle(s) $\theta_i$ of $C_i$ and $\beta_i$ (there are one or two such angles for each curve $C_i$) satisfy $\sin\theta_i\geq \rho$.
\end{lemma}

We shall need some background material on the conformal moduli of quadrilaterals. Recall that a quadrilateral $Q(z_1,z_2,z_3,z_4)$
consists of a Jordan domain
$Q$ in the complex plane and a sequence of ordered vertices $z_1,z_2,z_3,z_4$ on the boundary of $Q$.
The vertices of a quadrilateral $Q(z_1,z_2,z_3,z_4)$
divide its boundary into four Jordan arcs, called the sides of the quadrilateral. The arcs $\overline{z_1z_2}$ and
$\overline{z_3z_4}$ are called the $a$-sides and the other two arcs, the $b$-sides of $Q$. Two quadrilaterals $Q(z_1,z_2,z_3,z_4)$
and $Q'(w_1,w_2,w_3,w_4)$
are said to
be conformally equivalent if there is a conformal map from $Q$ to $Q'$ which carries each  $z_i$ to $w_i$.

Every quadrilateral $Q(z_1,z_2,z_3,z_4)$ is conformally equivalent to a
rectangle
\[R(0,a,a+ib,ib)=\{x+iy: 0<x<a, 0<y<b\}.\]
  It is easy to see that two rectangles $R(0,a,a+ib,ib)$ and $R'(0,a',a'+ib',ib')$
are conformally equivalent if and only if there is a similarity transformation between them. Therefore, we can define the (conformal) modulus of the quadrilateral $Q(z_1,z_2,z_3,z_4)$ by
$$\mathrm{mod}(Q(z_1,z_2,z_3,z_4))=\frac{a}{b}.$$
It follows from the definition that the modulus of a quadrilateral is a conformal invariant and that
$\mathrm{mod}(Q(z_1,z_2,z_3,z_4))=1/\mathrm{mod}(Q(z_2,z_3,z_4,z_1))$.

The modulus of a quadrilateral $Q(z_1,z_2,z_3,z_4)$ can be described in terms of extremal length
in the following way. Let $\mathcal{F}=\{\gamma\}$ be the family of curves in $Q$ joining the
$a$-sides. The extremal length of the family $\mathcal{F}$,
denoted by $\mathrm{Ext}(\mathcal{F})$, is defined by
$$\mathrm{Ext}(\mathcal{F}) = \sup_\rho \frac{{\inf_{\gamma\in \mathcal{F}}\ell_{\rho}(\gamma)}^2}{\mathrm{Area}_\rho}$$
 where the supremum is taken over all conformal metrics $\rho$ on $Q$ of finite
 positive area. Then it can be shown \cite{Ahlfors} that
 $$\mathrm{mod}(Q(z_1,z_2,z_3,z_4))=\frac{1}{\mathrm{Ext}(\mathcal{F})}.$$

\section{A lower bound for the quasiconformal dilatation under a twist}

The aim of this section is to prove Corollary \ref{coro:lower}, which is important for Section \ref{sec:main}. For a hyperbolic metric $X$ and a simple closed geodesic $\alpha$ on $X$,
we denote by $\tau^{t}_{\alpha}(X), t \in \mathbb{R}$ the hyperbolic metric obtained from $X$ by a Fenchel-Nielsen twist of magnitude $t$ along $\alpha$. We fix the simple closed curve $\alpha$ and, to simplify the notation, we set $X_t=\tau^{t}_{\alpha}(X)$.

For a small $\epsilon > 0$, let $N_\epsilon$ be an $\epsilon$-neighborhood of $\alpha$. We denote by $g_\epsilon^t:X  \rightarrow X_t$ a homeomorphism that is the natural isometry outside of $N_\epsilon$, and that is homotopic to the identity. 

It is convenient to work in the universal cover of the surface. In the case where $X$ has no boundary, its universal cover is the hyperbolic plane  $\mathbb{H}^2$. We let $f_\epsilon^t: \mathbb{H}^2 \to \mathbb{H}^2$ be a lift of $g_\epsilon^t$.
In the case where $X$ has non-empty boundary, we take the double of $X$ and $X_t$, extend the map
$g_\epsilon^t$ to $X^d \to X^d_t$ and we then let $f_\epsilon^t$ be the lift of the extended map to  $\mathbb{H}^2$.
Thus, in any case, the map $f_\epsilon^t$ is defined on the plane $\mathbb{H}^2$.

Let $\tilde{\alpha}$ be the lift of the closed geodesic $\alpha$ to the universal cover. Then $\tilde{\alpha}$ can be seen as a lamination with discrete leaves. 

When $\epsilon$ tends to zero, the maps $f_\epsilon^t$ coneverges pointwise on $\mathbb{H}^2 \setminus \tilde{\alpha}$ to a map $f^t$ that is an isometry on every connected component of $\mathbb{H}^2 \setminus \tilde{\alpha}$. We choose an orientation on $\alpha$. Now every connected component of $\tilde{\alpha}$ divides the plane into two parts, a left part and a right part. We extend $f^t$ to $\tilde{\alpha}$ by requiring that this extension is continuous on the left part. We denote the extended map again by $f^t$, a piecewise isometry from $\mathbb{H}^2$ to $\mathbb{H}^2$.  

To make some explicit computations, we work in the upper-half model of the hyperbolic plane $\mathbb{H}^2$. Up to conjugation in the domain and in the codomain, we can assume that the geodesic $i\mathbb{R}^+$ is a leaf of $\tilde{\alpha}$ that is fixed pointwise by $f^t$. In particular, $f^t$ fixes $0,i,\infty$.

\begin{lemma}\label{lemma:left} For any bi-infinite geodesic in the upper-half plane model of $\mathbb{H}^2$ with endpoints $ x_1<0<x_2$ on $\mathbb{R}$ and intersecting $i\mathbb{R}^+$ at $i$ (see Figure \ref{upper}), we have
$$ f^t(x_1)<-\sqrt{e^{2t}+(\frac{x_1+x_2}{2})^2}+(\frac{x_1+x_2}{2})<0 \hbox{ and  } 0< f^t(x_2)< x_2, \ \forall \ t> 0.$$
\end{lemma}
\begin{proof}

\bigskip
  \begin{figure}[!hbp]
\centering
\psfrag{x}{\small x}
\psfrag{y}{\small y}
\psfrag{i}{\small i}
\includegraphics[width=.40\linewidth]{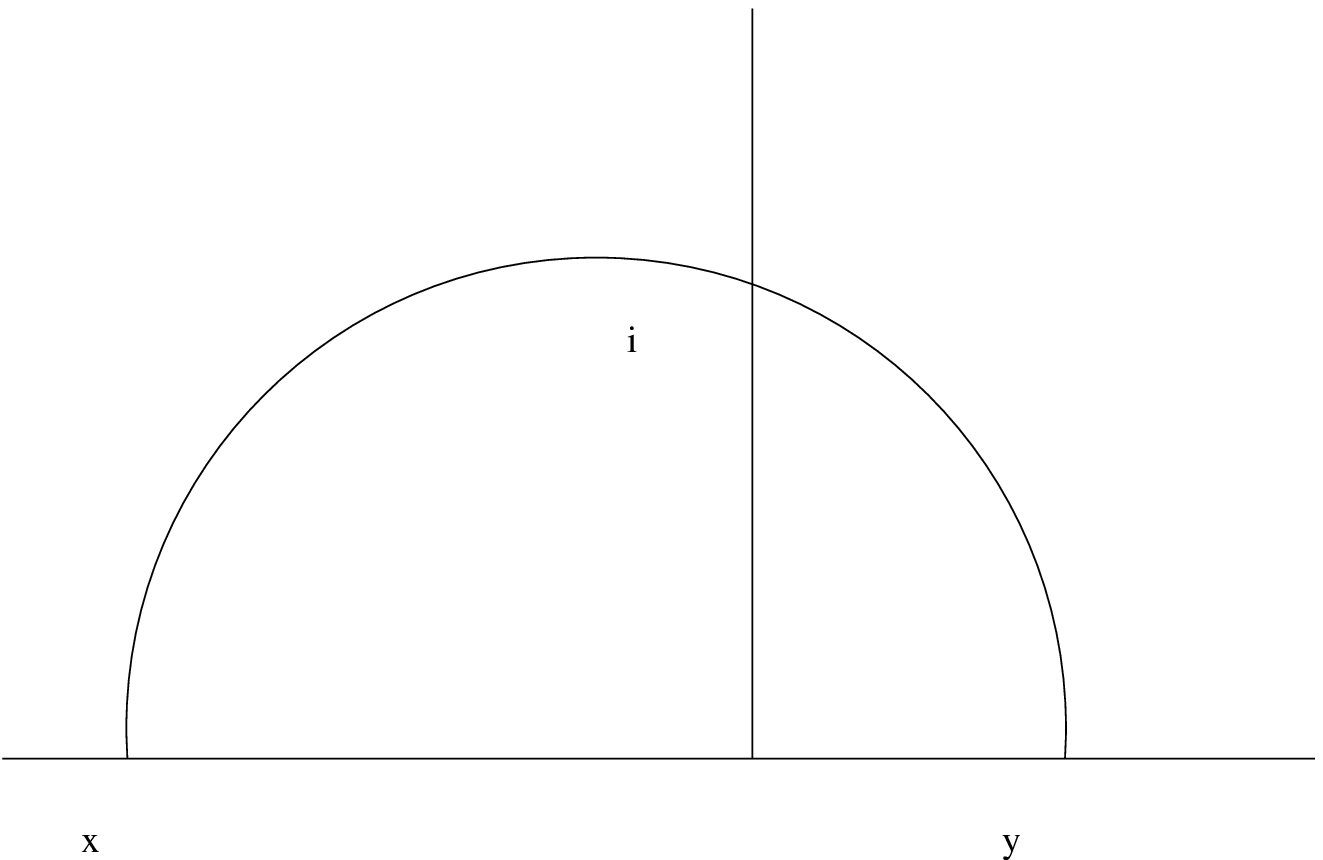}
\caption{\small {}}
\label{upper}
\end{figure}
\bigskip

This follows from the construction of the Fenchel-Nielsen twist deformation.
See, for example, the proof of Lemma 3.6 in Kerckhoff \cite{Kerckhoff}.  We give here the proof for the
sake of completeness.

Let $\gamma$ be the bi-infinite geodesic connecting $x_1$ and $x_2$.
By assumption, $i\mathbb{R^{+}}$  and $\gamma$ intersect
at the point $i$. Under the twist deformation,  $\gamma$ is deformed into
a sequence of disjoint geodesic arcs $\{A_i\}$, each coming from $\gamma$ under the twist deformation.
Let $\bar\gamma$ be the infinite piecewise geodesic arc, which is the union of $\{A_i\}$ and of pieces of
leaves of $\tilde{\alpha}$.
 See Figure \ref{arcs} in the case where $x_1=-1$ and $  x_2=1$.

\begin{figure}[!hbp]
\centering
\includegraphics[scale=0.60, bb=83 188 469 437]{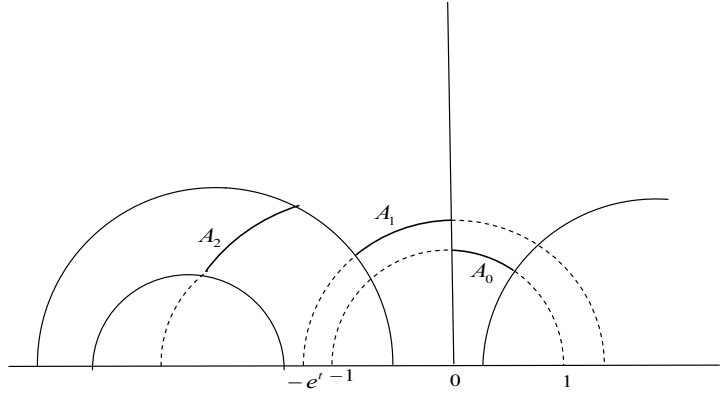}
\caption{\small{The image of $\gamma$ under a twist. }}
\label{arcs}
\end{figure}

Note that one such arc $A_0$ passes through the point $i$. If $A_0$ is continued to a bi-infinite geodesic, its endpoints
will be precisely those of $\gamma$. Move along $\bar{\gamma}$ in the left direction, running along the
 leaf $i\mathbb{R^{+}}$ (with hyperbolic distance $t$) until coming to the next arc $A_1$. If the arc is continued in
the forward direction, one of its endpoints is $-\sqrt{e^{2t}+(\frac{x_1+x_2}{2})^2}+(\frac{x_1+x_2}{2})$ (this can be shown by the cosine formula for triangles). Similarly, the forward endpoint of the next arc, $A_2$,
is strictly to the left of $-\sqrt{e^{2t}+(\frac{x_1+x_2}{2})^2}+(\frac{x_1+x_2}{2})$. In fact, the forward endpoint of each arc $A_{i+1}$ is strictly to the left of $A_i$.
Since the forward endpoints of the $A_i^{,}s$ converge to
$f^t(x_1)$, we see that $f^t(x_1)$ is strictly less than $-\sqrt{e^{2t}+(\frac{x_1+x_2}{2})^2}+(\frac{x_1+x_2}{2})$.

An analogous (and simpler) argument shows that $0<f^t(x_2)< x_2$.
\end{proof}

As before, assume that $\mathcal{P}=\{C_i\}$ is a geodesic pants decomposition of the
hyperbolic metric $R$.

Let $t=(t_1,t_2,\ldots)$ be a sequence of real numbers. Fix $X\in \mathcal{T}_{qc}(R)$.
We say that $X_t$ is a \emph{multi-twist deformation} of $X$ along $\mathcal{P}$ if $X_t$ is obtained from $X$ by the composition
of $t_i$-twists along $C_i$.
Let us set $\|t\|=\sup_i |t_i|$.
In the following theorem, the simple closed curves $(\beta_i)$ are chosen as in Section 2. We assume that the intersection angle(s) $\theta_i$ of
$C_i$ and $\beta_i$ satisfy $\sin\theta_i\geq \rho_i$.

Since each $\beta_i$ intersects $C_i$ and no $C_j, j \neq i$, then under the multi-twist deformation, the image of  $\beta_i$ depends only on the twist along $C_i$.

We now denote by $H(\infty,-1,0,t), t>0$ the quadrilateral where the Jordan domain $H$ is the upper half-plane.  Let $h(t)=\mathrm{mod}(H(\infty,-1,0,t))$. Then, $h(t)$ is related to the modulus of the Gr\"{o}tzch ring (the  ring domain obtained by deleting the interval $[0,r]$
from the unit disk)  $\mu(r)$ by the following equality (see Page 60--61 in \cite{Lehto}):
$$
h(t)=\frac{2}{\pi}\mu( {\sqrt{\frac{1}{1+\lambda}} }), \ \mathrm{where } \ \lambda=t.$$
From the known properties of  $\mu(r)$, it follows that $h(t)$ is a strictly increasing function and  $\lim\limits_{t \to +\infty}h(t)=\infty$.

\begin{theorem}\label{thm:lower}
 For the hyperbolic surface $X_t$ defined above, we have
$$d_{qc}(X,X_t)\geq  \frac{1}{2}\log \sup_i \frac{h(K_i e^{|t_i|})}
{h(\frac{1+\sqrt{1-\rho_i^2}}{1-\sqrt{1-\rho_i^2}})},$$
where
$$K_i=\frac{1-\sqrt{1-\rho_i^2}}{\left(1+\sqrt{1-\rho_i^2}\right)\left(\sqrt{1+(\frac{\sqrt{1-\rho_i^2}}{1-\sqrt{1-\rho_i^2}})^2}+\frac{\sqrt{1-\rho_i^2}}{1-\sqrt{1-\rho_i
^2}}\right)}.$$
\end{theorem}
\begin{proof}

For each $C_i$, consider the twist $t_i$. Without loss of generality, we assume that $t_i>0$.

Assume that $i\mathbb{R^{+}}$ is a lift of $C_i$ to the universal cover and that $\beta_i$ has a lift $\gamma$
which intersects $i\mathbb{R^{+}}$ at the point $i$. Denote by $x_1<0<x_2$ the two endpoints of $\gamma$.
The intersection $\theta$ of $i\mathbb{R^{+}}$ and $\gamma$ satisfies:
$$\sin\theta= \frac{2}{|x_1|+|x_2|}.$$

Applying Lemma \ref{lemma:left}, we have
\begin{equation}\label{equ:mix4}
f^t(x_1)<-\sqrt{e^{2t_i}+(\frac{x_1+x_2}{2})^2}+(\frac{x_1+x_2}{2})<0,  0<f^t(x_2)<x_2.
\end{equation}

Let us set $-\sqrt{e^{2t_i}+(\frac{x_1+x_2}{2})^2}+(\frac{x_1+x_2}{2})=-Ae^{t_i}$, where
$$A=\frac{1}{\sqrt{1+e^{-2t_i}(\frac{x_1+x_2}{2})^2}+e^{-t_i}(\frac{x_1+x_2}{2})}.$$

By the geometric definition of quasiconformal maps,
\[K(f^t) \geq \frac{\mathrm{mod}\left(H(f^t(x_1),f^t(0),f^t(x_2),f^t(\infty))\right)}{\mathrm{mod}(H(x_1,0,x_2,\infty)}.\]
By $(\ref{equ:mix4})$ and the monotonicity of conformal modulus, we have
$$\mathrm{mod}(H(f^{t}(x_1),0,f^t(x_2),\infty)) \geq  \mathrm{mod}(H(-Ae^{t_i},0,x_2,\infty)).$$
Note that
\begin{eqnarray*}
\mathrm{mod}(H(-Ae^{t_i},0,x_2,\infty))&=&\mathrm{mod}(H(\infty,-x_2,0,Ae^{t_i})) \\
&=&\mathrm{mod}(H(\infty,-1,0,\frac{Ae^{t_i}}{|x_2|}).
\end{eqnarray*}
Therefore
\begin{equation}\label{equ:lower}K(f^t)\geq \frac{\mathrm{mod}(H(\infty,-1,0,\frac{Ae^{t_i}}{|x_2|}))}{\mathrm{mod}(H(\infty,-1,0,|\frac{x_1}{x_2}|))}.
\end{equation}

We can use the cross ratio to estimate $|\frac{x_1}{x_2}|$ in terms of the lower bound $\rho_i(X)$ of $\sin \theta_i$. Let $\chi(a,b,c,d)=\frac{(a-c)(b-d)}{(a-d)(b-c)}$
be the cross ratio of
$a,b,c,d\in \mathbb{R}\cup \{\infty\}$.
We can map $\mathbb{H}^2$ conformally to the unit disc, and  $i\mathbb{R}^+$ and $\gamma$ to the geodesics with endpoints $\pm 1$ and $\pm e^{i\theta}$
respectively.
It is easy to show that
\begin{equation*}\label{eq:cos}
\cos^2(\theta/2)=\chi(1,e^{i\theta},-e^{i\theta},-1).
\end{equation*}
By the conformal invariance of the cross ratio, we have
$$\cos^2(\theta/2)=\chi(0,x_2,x_1,\infty)=\frac{|x_1|}{|x_1|+|x_2|}.$$
Since $\sin\theta \geq \rho_i$, we have
\begin{equation}\label{eq:bound}
\frac{1-\sqrt{1-\rho_i^2}}{1+\sqrt{1-\rho_i^2}} \leq \frac{|x_1|}{|x_2|}\leq \frac{1+\sqrt{1-\rho_i^2}}{1-\sqrt{1-\rho_i^2}}.
\end{equation}

There are other restrictions on the values $x_1, x_2$. Since the geodesic $\gamma$ passes through  the point $i$, we can show that $|x_2||x_1|=1$.

If $|x_1|\geq 1 \geq |x_2|$, then
$$\frac{A}{|x_2|}\geq A\geq \frac{1}{\sqrt{1+(\frac{|x_1|/|x_2|-1}{2})^2}+\frac{|x_1|/|x_2|-1}{2}}.$$
Combined with the inequality $(\ref{eq:bound})$, we obtain
$$\frac{A}{|x_2|}\geq \frac{1}{\sqrt{1+(\frac{\sqrt{1-\rho_i^2}}{1-\sqrt{1-\rho_i^2}})^2}+\frac{\sqrt{1-\rho_i^2}}{1-\sqrt{1-\rho_i
^2}}}.$$

If $|x_2|\geq 1 \geq |x_1|$, then (using again Inequality $(\ref{eq:bound})$) we obtain
$$\frac{A}{|x_2|}\geq A \frac{|x_1|}{|x_2|}\geq \frac{1-\sqrt{1-\rho_i^2}}{\left(1+\sqrt{1-\rho_i^2}\right)\left(\sqrt{1+(\frac{\sqrt{1-\rho_i^2}}{1-\sqrt{1-\rho_i^2}})^2}+\frac{\sqrt{1-\rho_i^2}}{1-\sqrt{1-\rho_i
^2}}\right)}.$$

Denote $$K_i=\frac{1-\sqrt{1-\rho_i^2}}{\left(1+\sqrt{1-\rho_i^2}\right)\left(\sqrt{1+(\frac{\sqrt{1-\rho_i^2}}{1-\sqrt{1-\rho_i^2}})^2}+\frac{\sqrt{1-\rho_i^2}}{1-\sqrt{1-\rho_i
^2}}\right)}.$$

We conclude that
$$\frac{A}{|x_2|}\geq K_i.$$

It follows from inequality $(\ref{equ:lower})$ that
$$K(f^t) \geq \frac{\mathrm{mod}(H(\infty,-1,0,K_ie^{t_i}))}
{\mathrm{mod}(H(\infty,-1,0,\frac{1+\sqrt{1-\rho_i^2}}{1-\sqrt{1-\rho_i^2}}))}.$$

Using the function $h(t)$ already introduced for moduli of quadrilaterals, we get
$$K(f^t) \geq \frac{h( K_ie^{t_i})}
{h(\frac{1+\sqrt{1-\rho_i^2}}{1-\sqrt{1-\rho_i^2}})}.$$

By Teichm\"uller's Theorem,
there is an extremal quasiconformal map from $X$ to $X_t$ that realizes the
Teichm\"uller distance $d_{\mathrm{qc}}(X_t,X)$. Lifting this map to the universal cover, since it is homotopic
to $f^t$, it has the same boundary value as $f^t$.  It follows that
$$2 d_T(X,X_t)\geq \log \sup_i \frac{h( K_ie^{t_i})}
{h(\frac{1+\sqrt{1-\rho_i^2}}{1-\sqrt{1-\rho_i^2}})}.$$

\end{proof}

\begin{corollary}\label{coro:lower}  If $\sup_i { l_X(C_i)}< \infty$ and if $X_t$ is a multi-twist deformation of $X$ along $\mathcal{P}$, then
$d_{qc}(X,X_t) \to \infty$ as $\|t\| \to \infty$.
\end{corollary}
\begin{proof}
By assumption, there is a positive constant $M$ such that $\sup_i { \ell_X(C_i)}< M$. By Lemma
\ref{dual}, there is a positive constant $\rho$ (depending on $M$)  such that
$\inf_i \rho_i \geq \rho$. It is easy to see that $\frac{(1-\sqrt{1-\rho^2})^2}{1+\sqrt{1-\rho^2}}$  is an increasing function of $\rho$.
For each $i$,  we  have
$$\frac{(1+\sqrt{1-\rho_i^2})^2}{1-\sqrt{1-\rho_i^2}} \geq \frac{(1-\sqrt{1-\rho^2})^2}{1+\sqrt{1-\rho^2}}. $$

As a result, if follows from Theorem \ref{thm:lower} that
\begin{equation}\label{eq:main}
d_{qc}(X,X_t)\geq  \frac{1}{2}\log \sup_i \frac{h( \frac{(1-\sqrt{1-\rho^2})^2}{1+\sqrt{1-\rho^2}}e^{t_i})}
{h(\frac{1+\sqrt{1-\rho^2}}{1-\sqrt{1-\rho^2}})}.
\end{equation}
As $\|t\| \to \infty$, the properties of the function $h(t)$ tell us that
$d_{qc}(X,X_t) \to \infty$.
\end{proof}

\section{Estimation of hyperbolic length under a twist deformation}

The twist deformation is an important tool to understand the difference between the quasiconformal metric and the length-spectrum metric.
As in the previous section, we fix a simple closed geodesic $\alpha$ and we set $X_t=\tau^{t}_{\alpha}(X)$.
In this section, we give a upper bound  and a lower bound for $d_{ls}(X, X_t)$.

\begin{proposition}\label{prop:ls-upper2} For every $t$ in $\mathbb{R}$, we have
$$d_{ls}(X,X_t) \le \frac{1}{2} \max\left\{\sup_{\gamma, i(\alpha, \gamma)\ne 0}\frac{i(\alpha,\gamma)|t|}{\ell_X(\gamma)}, \sup_{\gamma, i(\alpha, \gamma)\ne 0}\frac{i(\alpha,\gamma)|t|}{\ell_{X_t}(\gamma)} \right\}.$$
\end{proposition}
\begin{proof}
Without loss of generality, we can assume that $t>0$.
For any simple closed curve $\gamma$ intersecting $\alpha$, let $\ell_t(\gamma)$
denote the hyperbolic length of $\gamma$ in
$X_t$. We have
$$\ell_X(\gamma)-i(\alpha,\gamma)t \leq \ell_t(\gamma)\leq \ell_X(\gamma)+i(\alpha,\gamma)t.$$
Recall that  the length-spectrum distance is given by
$$d_{ls}(X,X_t) = \max \left\{\frac{1}{2}\log
\sup_\gamma\frac{\ell_{t}(\gamma)}{\ell_{X}(\gamma)},\frac{1}{2}\log
\sup_\gamma\frac{\ell_{X}(\gamma)}{\ell_{t}(\gamma)}\right\},$$
where the supremum is taken over all essential simple closed curves.

For a simple closed curve $\gamma$ satisfying $i(\alpha, \gamma)=0$,
the hyperbolic length of $\gamma$ is invariant under the twist along $\alpha$.
As a result,  we have
$$d_{ls}(X,X_t)
= \max \left\{\frac{1}{2}\log
\sup_{\gamma, i(\alpha, \gamma)\ne 0}\frac{\ell_{t}(\gamma)}{\ell_{X}(\gamma)},\frac{1}{2}\log
\sup_{\gamma, i(\alpha, \gamma)\ne 0}\frac{\ell_{X}(\gamma)}{\ell_{t}(\gamma)}\right\}.
$$

For any simple closed curve $\gamma$ with $i(\alpha, \gamma)\ne 0$, we have
$$\log\frac{\ell_{t}(\gamma)}{\ell_{X}(\gamma)} \le \log \frac{\ell_X(\gamma)+i(\alpha,\gamma)t}{\ell_X(\gamma)} \leq \frac{i(\alpha,\gamma)t}{\ell_X(\gamma)}$$
(using, for the right-hand side, the inequality $\log (1+x)\leq x$ for $x>0$),
and likewise
$$\log\frac{\ell_{X}(\gamma)}{\ell_{t}(\gamma)} \le\left|\log \frac{\ell_t(\gamma)+i(\alpha,\gamma)t  } { \ell_t(\gamma)}\right|\leq \frac{i(\alpha,\gamma)t}{\ell_X(\gamma)}.$$
The result is thus proved. 
\end{proof}
Note that if $\ell_X(\alpha)\leq L$, then it follows from the Collar Lemma that
there is a constant $C$ depending on $L$ such that for
any simple closed geodesic $\gamma$ with $i(\alpha, \gamma)\ne 0$, we have
$\ell_X(\gamma)\geq C i(\alpha,\gamma)|\log\ell_X(\alpha)|$ and $\ell_{X_t}(\gamma)\geq C i(\alpha,\gamma)|\log\ell_X(\alpha)|$.
We deduce from the Proposition \ref{prop:ls-upper2} the following
\begin{corollary}\label{cor:ls-upper}
If $\ell_X(\alpha)\leq L$, then there is a constant $C$ depending on $L$ such that
$$d_{ls}(X,X_t) \le
\frac{|t|}{2C|\log \ell_X(\alpha)|}. $$
\end{corollary}
\bigskip
Now we need a lower bound. We use the idea of an $(\epsilon_0, \epsilon_1)$-decomposition of a hyperbolic surface (Minsky \cite[sec. 2.4]{Minsky},  Choi-Rafi \cite[sec. 3.1]{CR}).

Consider a hyperbolic metric $X$ on a surface of finite type (we will need only the case when $X$ is homeomorphic to a one-holed torus or to a four-holed sphere). 

Choose two numbers $\epsilon_1 < \epsilon_0$ less than a Margulis constant of the hyperbolic plane. Assume that $\alpha$ is a closed geodesic in the interior of $X$ with $\ell_X(\alpha) \leq \epsilon_1$. Let $A$ be an annular (collar) neighborhood of $\alpha$ such that the two boundary components of $A$ have length $ \epsilon_0$. We can choose $\epsilon_1$ and $  \epsilon_0$ small enough such that any simple closed geodesic on $X$ that intersects $\alpha$ is either the core curve of $A$ or crosses $A$ (this upper bound for $\epsilon_1, \epsilon_0$ can be chosen in a way that is independent on the surface $X$).

Let $Q=X- A$. For any simple closed geodesic $\gamma$ on $X$, its restriction to $Q$  is homotopic (relative to $\partial Q$ ) to a shortest geodesic, which we denote by $\gamma_Q$.

\begin{lemma} [ {Choi-Rafi \cite[prop. 3.1]{CR}} ]     \label{lem:CR}
There is a constant $C$
depending on $\epsilon_0, \epsilon_1$ and on the topology of $X$, such that
$$|\ell_X(\gamma\cap Q)-\ell_X(\gamma_Q)| \leq C i(\gamma, \partial Q),$$
$$|\ell_X(\gamma\cap A)-[2\log \frac{\epsilon_0}{\ell_X(\alpha)}+\ell_X(\alpha)|\mathrm{tw}_X(\gamma, \alpha)| ] i(\gamma, \alpha)| \leq C i(\gamma, \alpha).$$
\end{lemma}
In the second formula, the quantity $\mathrm{tw}_X(\gamma, \alpha)$ is called the twist of $\gamma$ around $\alpha$. This quantity is defined in \cite[sec. 3]{Minsky}. Its difference with the Fenchel-Nielsen twist coordinate  is given by
the following estimate (For the proof, see Minsky \cite[Lemma 3.5]{Minsky}).
\begin{lemma}[ {Minsky \cite[lemma 3.5]{Minsky}} ]\label{lem:M}
Suppose that $X_t$ is the twist deformation of $X$ by a twist of magnitude $t=\tau_{X_t}(\alpha)-\tau_{X}(\alpha)$ along $\alpha$.
We normalize the twist coordinate by setting $s(X_t)=\frac{\tau_{X_t}(\alpha)}{\ell_X(\alpha)} $ and $s(X)=\frac{\tau_{X}(\alpha)}{\ell_X(\alpha)} $.
Then
$$|\mathrm{tw}_{X_t}(\gamma, \alpha)-\mathrm{tw}_X(\gamma, \alpha)-(s(X_t)-s(X))|\leq 4.$$
\end{lemma}

Now assume that $X$ is homeomorphic to a one-holed torus or to a four-holed sphere, with the pair of pants decomposition $\alpha$. Consider the simple closed curve $\beta$ we constructed in Lemma \ref{dual}. %First we cut $X$ along the geodesic $\alpha$ into a pair of pants. Let $\gamma$ be the geodesic arc that intersects perpendicularly $\alpha_1, \alpha_2$ (the geodesic boundary components of the pair of pants coming from $\alpha$). By choosing a segment of $\alpha$ connecting the two intersection points and joining it with $\gamma$, we get $\beta$.
Note that $ |\mathrm{tw}_X(\beta, \alpha)|$ is less than $4$, and $i(\beta,\alpha)$ is $1$ or $2$.
The following is a direct corollary of Lemma \ref{lem:CR} and \ref{lem:M}. 
\begin{lemma} \label{lem:beta}
If $X_n$ is the hyperbolic metric obtained from $X$ by a twist deformation along $\alpha$,
with normalized twist coordinates $n=\frac{\tau_{X_n}(\alpha)-\tau_{X}(\alpha)}{\ell_X(\alpha)}$, then
$$\frac{\ell_{X_n}(\beta)}{\ell_X(\beta)} \geq \frac{2\log \frac{\epsilon_0}{\ell_X(\alpha)}+\ell_X(\alpha)(n-8)+\frac{1}{2}\ell_{X_n}(\beta_Q)-2C}{2\log \frac{\epsilon_0}{\ell_X(\alpha)}+2\ell_X(\alpha)+\ell_X(\beta_Q)+2C},$$
where $C$ is the same constant as in Lemma \ref{lem:CR}, hence it only depends on $\epsilon_0, \epsilon_1$ and on the topology of $X$.
\end{lemma}
\begin{proof}
We know that $i(\beta,\alpha)$ and $i(\beta,\partial Q)$ are $1$ or $2$, and that $ |\mathrm{tw}_X(\beta, \alpha)| \leq 4$. By Lemma \ref{lem:M} we see that $n-8 \leq |\mathrm{tw}_X(\beta, \alpha)| \leq n+8$. We also know that $\ell_{X_n}(\alpha) = \ell_{X}(\alpha)$. We can express $\ell_X(\beta)$ as $\ell_X(\beta \cap Q) + \ell_X(\beta \cap A)$, and we do the same for $\ell_{X_n}(\beta)$. We then use Lemma \ref{lem:CR} to estimate separately $\ell_X(\beta \cap Q), \ell_X(\beta \cap A), \ell_{X_n}(\beta \cap Q), \ell_{X_n}(\beta \cap A)$. Finally, we estimate the ratio, and we get the formula.
\end{proof}

%For a hyperbolic surface homeomorphic to a four-holed sphere, we have a similar result.

Consider now a hyperbolic metric $X$ with a geodesic pants decomposition $\mathcal{P}=\{{C_i}\}$ satisfying $\sup_{C_i} \ell_X(C_i)\leq M$. Suppose there exists a geodesic $\alpha\in \mathcal{P}$ which is  in the interior of $X$ and which has length less than $ \epsilon_1$. Let $X_t=\tau^{t}_{\alpha}(X)$. Here we set $t=\ell_R(C_i)n$, for a sufficiently large number $n$. With these assumptions, we now apply Lemma \ref{lem:beta} to give a lower bound for the length-spectrum distance $d_{ls}(X,X_t)$.

As we have done before, choose a simple closed curve $\beta$ which intersects $\alpha$ once
or twice but does not intersect any other curves in $\mathcal{P}$.  Under the upper-boundedness  assumption on the pants decomposition, there is also a constant $K$ depending on $\epsilon_0, \epsilon_1, M$ such that the length $\ell_X(\beta_Q)$ is bounded as
$$1/K \leq \ell_X(\beta_Q) \leq K.$$ 
Moreover, we may choose the constant $\epsilon_0$ with upper and lower bounds which only depend on $M$.
Then an analysis of the formula in Lemma \ref{lem:beta} leads to the following
\begin{theorem} \label{thm:beta}
If $\sup_{C_i} \ell_X(C_i)\leq M$ and $\ell_X(\alpha)\leq \epsilon_1$,
then there is a constant $D$ depending on $ \epsilon_1, M$ such that
\begin{eqnarray*}d_{ls}(X,X_t)&\geq&  \frac{1}{2}\log\frac{\ell_{X_t}(\beta)}{\ell_X(\beta)}  \\
&\geq& \frac{1}{2}\log\frac{2|\log \ell_X(\alpha)|+|t|-D}{2|\log \ell_X(\alpha)|+D}.
\end{eqnarray*}
\end{theorem}
\begin{proof}
We can choose $\epsilon_0$ and $\epsilon_1$ such that they satisfy $\epsilon_0 = 2 \epsilon_1$. Then $\log(\epsilon_0)$ can be put in the constant. We remark that the constant $\epsilon_0$ is less than the Margulis constant that is less than $1$. In particular $\ell_X(\alpha)$ is less than $1$ and it can be put in the constant. The terms $\ell_X(\beta_Q)$ and $\ell_{X_n}(\beta_Q)$ can be estimated with $K$ as above. Moreover, $n \ell_X(\alpha) = t$. Transforming the formula of Lemma \ref{lem:beta} in this way, we get the conclusion.
\end{proof}

In fact, Theorem \ref{thm:beta} is a particular case of Choi-Rafi's product region formula for the length-spectrum metric. Their proof requires a more detailed and complicated analysis, see Lemma 3.4 and Theorem 3.5 of \cite{CR}.  Since we only need to consider a particular curve $\beta$ to give a lower bound, we don't need the general formula.

\section{Structure on the length-spectrum Teichm\"uller space} \label{sec:main}

The goal of this section is to show how bad can be the inclusion $(T_{qc},d_{qc}) \to (T_{ls},d_{ls})$. In Section \ref{sec:connected}, we will use some of the results of this section to prove connectedness of $(T_{ls},d_{ls})$ 

Given a surface $S$, we define a proper subsurface $S'$ of $S$ to be an open subset of $S$ such that the frontier $\partial S$ is a union of simple closed essential curves of $S$. If $X$ is a complex structure on $S$ and if $S'$ is a proper subsurface of $S$, the restriction of $X$ to $S'$ is the complex structure that we have on the proper subsurface $X'$ that is homotopic to $S$ and such that the boundary curves of $X'$ are geodesics for the intrinsic hyperbolic metric of $X$. The intrinsic metric of $X'$ is the restriction of the intrinsic metric of $X$.     

\begin{proposition}\label{pro:exhau}
Consider an exhaustion of $S$ by a sequence of subsurfaces with boundary:
$$S_1 \subseteq S_2 \subseteq \cdots \subseteq S_n \subseteq \cdots
\ \mathrm{and} \ S=\cup_{n=1}^\infty S_n.$$ Given two complex structures $X, Y$ on $S$, let $X_n$ and $Y_n$ be the complex structures obtained by restriction of $X$ and
$Y$ to $S_n$ respectively. Then
\begin{equation}\label{equ:exhau}
d_{ls}(X,Y)= \lim_{n\to\infty}d_{ls}(X_n,Y_n).
\end{equation}
\end{proposition}
\begin{proof}
From the definition, for any $\epsilon>0$, there exists a simple closed curve $\gamma$ on $\Sigma$ such that
$$d_{ls}(X,Y)< \frac{1}{2}\left|\log \frac{\ell_{X}(\gamma)}{\ell_{Y}(\gamma)}\right|+\epsilon.$$
Such a curve $\gamma$ must lie in some subsurface $\Sigma_{n_0}$. As a result,
$$d_{ls}(X,Y)< d_{ls}(X_{n_0},Y_{n_0})+\epsilon.$$
Since $d_{ls}(X_n,Y_n)$ is increasing  and $\epsilon$ is arbitrary, we conclude that
$$d_{ls}(X,Y)\leq \lim_{n\to\infty}d_{ls}(X_n,Y_n).$$
The other side of $(\ref{equ:exhau})$ is obvious.
\end{proof}

Let $R$ be a hyperbolic surface of infinite type with a geodesic pants decomposition $\mathcal{P}=\{C_i\}$. In the rest of this section, we always assume that the pair $(R,\mathcal{P} )$ is upper-bounded and that it admits short interior curves, i.e. that there is a sequence of curves of the decomposition $\alpha_k = C_{i_k}$ contained in the interior of $R$ with $\ell_R(\alpha_i) \to 0$. Note that for any point $X\in \mathcal{T}_{ls}(R)$, the geodesic representative of $\mathcal{P}$ satisfies the same upper-boundedness property and admits short interior curves.

The following result was proved in \cite{LSW} (see also \cite[Example 5.1]{ALPS1}). The proof here is simpler.
\begin{proposition} \label{prop:inv}
Under the above assumptions, the inverse of the inclusion map $I$ restricted to $\mathcal{T}_{qc}(R)$:  $(\mathcal{T}_{qc}(R), d_{ls}) \to (\mathcal{T}_{qc}(R), d_{qc})$ is nowhere continuous. More precisely, for any $X\in  \mathcal{T}_{qc}(R)$, there is a sequence $X_n\in \mathcal{T}_{qc}(R)$ with $d_{qc}(X,X_n) \to \infty$ while $d_{ls}(X,X_n) \to 0$.
\end{proposition}

\begin{proof}
Let $X\in \mathcal{T}_{qc}(R)$. Assume that $\ell_X(\alpha_n)=C_{i_n}=\epsilon_n \to 0$. Let $X_n=\tau_{\alpha_n}^{t_n}(X)$, with
$t_n= \log | \log \epsilon_n| \to \infty$.

By Corollary \ref{cor:ls-upper},
$$d_{ls}(X,X_n)\le
\frac{\log |\log \epsilon_n|}{2C|\log \epsilon_n|},$$
which tends to $0$ as $n\to \infty$.  On the other hand, Corollary \ref{coro:lower} shows that
$d_{qc}(X,X_n) \to \infty$.
\end{proof}
\begin{remark}
There exist hyperbolic surfaces $R$ with no pants decomposition  satisfying Shiga's property, but where the space
$(\mathcal{T}_{qc}(R), d_{qc})$  is topologically equivalent to the space $(\mathcal{T}_{qc}(R), d_{ls})$, see
Kinjo \cite{Kinjo}.  It would be interesting to know that whether the two metrics in Kinjo's examples are locally bi-Lipschitz.
\end{remark}

\begin{proposition}[Boundary point]\label{pro:boundary}
Under the above assumptions, there exists  a point in $\mathcal{T}_{ls}(R)\setminus \mathcal{T}_{qc}(R)$ that can be approximated by a sequence in $\mathcal{T}_{qc}(R)$ with the length-spectrum metric.
\end{proposition}
\begin{proof}
Let $X\in \mathcal{T}_{qc}(R)$.  By assumption, there exists a sequence of simple closed curves $\alpha_n$ such that $\ell_X(\alpha_n)=\epsilon_n \to 0$. Let $X_n=\tau_{\alpha_n}^{t_n}(X_{n-1})=\tau^{t_n}_{\alpha_n} \circ \cdots\circ \tau^{t_2}_{\alpha_2}  \circ\tau^{t_1}_{\alpha_1}(X)$, with $t_n= \log | \log \epsilon_n| \to \infty$. Define $X_\infty$ as the surface obtained from $X$ by a twist of magnitude $t_i$ along $\alpha_i$ for every $i$.

By  Inequality $(\ref{eq:main})$,
$d_{qc}(X,X_n) \to \infty$ and $d_{qc}(X,X_\infty)=\infty$.

For any simple closed curve $\gamma$ on  $X$, by an argument similar to that of the proof of Proposition \ref{prop:ls-upper2} and Corollary \ref{cor:ls-upper}, we have
$$d_{ls}(X, X_\infty) \le
\sup_{\gamma\in \mathcal{S}}\frac{\sum_{n=1}^\infty i(\gamma, \alpha_n)\log |\log \epsilon_n|}{2C\sum_{n=1}^\infty i(\gamma, \alpha_n) |\log \epsilon_n|}. $$

To see that the right hand side is uniformly bounded (independently of $\gamma$), we can use the inequality
$$\frac{\sum_{n=1}^\infty x_i}{\sum_{n=1}^\infty y_i}\leq \sum_{n=1}^\infty\frac{x_i}{y_i},$$
that holds for positive values of $x_i$ and $y_i$. To see this just note that all terms of $\sum_{n=1}^k x_i$ also appear in the product $\displaystyle \sum_{n=1}^k\frac{x_i}{y_i} \sum_{n=1}^k y_i$.
With this inequality we can see that
$$\frac{\sum_{n=1}^\infty i(\gamma, \alpha_n)\log |\log \epsilon_n|}{\sum_{n=1}^\infty i(\gamma, \alpha_n) |\log \epsilon_n|}\leq \sum_{n=1}^\infty\frac{ \log |\log \epsilon_n|}{ |\log \epsilon_n|}.$$
Moreover, the last sum is controlled by $\int_{\log |\log \epsilon_1|}^\infty e^{-x}xdx$,
which is bounded.

Note that
$$d_{ls}(X_n, X_\infty) \le
\sup_{\gamma\in \mathcal{S}} \frac{\sum_{k=n+1}^\infty i(\gamma, \alpha_k)\log |\log \epsilon_k|}{2C\sum_{k=n+1}^\infty i(\gamma, \alpha_k) |\log \epsilon_n|}. $$ Since the interval $\int_{\log |\log \epsilon_n|}^\infty e^{-x}xdx$ tends to zero as $n$ tends to infinity, we have
$d_{ls}(X_n, X_\infty)\to 0$.
\end{proof}
The above proposition shows that $d_{ls}$ is not complete when restricted to $\mathcal{T}_{qc}(R)$. This result was already obtained in \cite{ALPS1} and the proof we give here is more direct.

\begin{proposition}[Nowhere open]\label{prop:no}
With the assumptions above, in the metric space $(\mathcal{T}_{ls} (R), d_{ls})$, any open neighborhood of a point $X \in \mathcal{T}_{qc} (R)$ contains a point in $\mathcal{T}_{ls}(R)\setminus \mathcal{T}_{qc} (R)$.
\end{proposition}
\begin{proof}
As above, fix $t_k= \log | \log \epsilon_k| \to \infty$ and assume that $\ell_X(\alpha_n)=\epsilon_n \to 0$. We let $Y_n$ be the surface obtained from $X$ by a twist of magnitude $t_i$ along $\alpha_i$ for every $i \geq n$.
%$Y_n=\lim_{k \to\infty} \tau^{t_k}_{\alpha_k} \circ \cdots\circ \tau^{t_2}_{\alpha_{n+1}}  \circ\tau^{t_n}_{\alpha_n}(X)$
Then we can apply the proof of Proposition  \ref{pro:boundary}.
\end{proof}

Proposition \ref{prop:no} is also a consequence of Theorem \ref{th:W} below.

%Now we can give the proof of Theorem  \ref{thm:main}.

\begin{theorem}  \label{thm:main}
With the assumptions above, the space $(\mathcal{T}_{qc}(R),d_{ls} )$ is not dense in the space $(\mathcal{T}_{ls}(R),d_{ls} )$.
\end{theorem}
\begin{proof}
We will show that there exists a point in $\mathcal{T}_{ls}(R)$ which is not a limit point of
$ \mathcal{T}_{qc} (R)$ with the length-spectrum metric.

Start with a point $X \in  \mathcal{T}_{qc} (R)$. Assume that $\ell_X(\alpha_n)=\epsilon_n \to 0$. Let $Y_n=\tau_{\alpha_n}^{T_n}(X_{n-1})=\tau^{T_n}_{\alpha_n} \circ \cdots\circ \tau^{T_2}_{\alpha_2}  \circ\tau^{T_1}_{\alpha_1}(X)$, with $T_n= N | \log \epsilon_n| \to \infty$ where $N$ is a fixed positive constant. In this proof, we could have taken $N=1$, but taking a general $N$ is important for the proof of Proposition \ref{pro:nodense} below. 
Define $Y_\infty$ as the surface obtained from $X$ by a twist of magnitude $T_i$ along $\alpha_i$ for every $i$. 
It is not hard to see that $Y_\infty\in  \mathcal{T}_{ls} (R)$ and, in fact,
\begin{equation}\label{equ:N}
d_{ls}(X,Y_\infty)\leq \frac{N}{2C}.
\end{equation}

Suppose that there is a sequence of $X_k\in\mathcal{T}_{qc} (R)$, such that $d_{ls}(Y_\infty, X_k)\to 0$ as $k\to \infty$.  Denote the difference of the twist coordinates of each $X_k$ from $X$ (with respect to the pants decomposition $\mathcal{P}$) by $(\tau_k(C_i))$. Since the pants decomposition is upper-bounded and $X_k\in\mathcal{T}_{qc} (R)$, we have $\sup_i |\tau_k(C_i)| < \infty$, since otherwise, Corollary \ref{coro:lower} would imply $d_{qc}(X, X_k)=\infty $.

Note that for each $X_k$, its difference of twist coordinates with $Y_\infty$ is equal to
\begin{enumerate}
\item  $T_n-\tau_k(\alpha_n)$, for each $\alpha_n$;
\item  $-\tau_k(C_i)$, for each $C_i\in\mathcal{P}\setminus\{\alpha_n\}$.
\end{enumerate}
For each $\alpha_n$, choose a simple closed curves $\beta_n$ as in the construction before
Theorem  \ref{thm:beta}; then, by theorem \ref{thm:beta} we have
\begin{eqnarray*}
d_{ls}(X_k,Y_\infty ) &\geq& \sup_n\frac{1}{2}\log\frac{(N+2)|\log \epsilon_n|-C-\tau_k(\alpha_n)}{2|\log \epsilon_n|+C}.
\end{eqnarray*}
The right hand side of the above inequality has a positive lower bound that is independent of $k$ (because $\tau_k(\alpha_n)$ is bounded).
As a result, the sequence $X_k$ cannot approximate $Y_\infty$ in the length-spectrum metric.
\end{proof}

Denote the closure of $\mathcal{T}_{qc} (R)$ in $\mathcal{T}_{ls} (R)$ (with the length-spectrum metric) by $\overline{\mathcal{T}_{qc} (R)}$.  By Theorem \ref{thm:main}, $\mathcal{T}_{ls} (R) - \overline{\mathcal{T}_{qc} (R)}$  is not empty.
\begin{proposition}\label{pro:nodense}
Under the assumptions above, for any $X\in \mathcal{T}_{qc} (R)$, there is a sequence of points $Z_k\in \mathcal{T}_{ls} (R) - \overline{\mathcal{T}_{qc} (R)}$, such that $d_{ls} (X,Z_k)\to 0$.
\end{proposition}
\begin{proof}
For any $X\in \mathcal{T}_{qc} (R)$, we let $Z_k$ be the same as $Y_\infty$, which we constructed in the proof of Theorem \ref{thm:main}, by setting $N=\frac{1}{k}$ for each $k$.

As we have shown before, $Z_k\in \mathcal{T}_{ls} (R) - \overline{\mathcal{T}_{qc} (R)}$. Moreover, by inequality $(\ref{equ:N})$, we have $$d_{ls} (X,Z_k)\leq \frac{1}{2Ck},$$
which tends to $0$ as $k$ tends to $\infty$.
\end{proof}

Using the previous proposition, we can prove now the following result:
\begin{theorem}\label{th:W}
Under the assumptions above, the space $(\mathcal{T}_{qc}(R), d_{ls})$ is nowhere dense in $(\mathcal{T}_{ls}(R), d_{ls})$.
\end{theorem}
\begin{proof}
This is equivalent to prove that $\overline{\mathcal{T}_{qc} (R)}$ has no interior point. Consider an arbitrary $Y\in \overline{\mathcal{T}_{qc} (R)}$ and an arbitrary $\epsilon>0$. If $Y\in \partial\overline{\mathcal{T}_{qc} (R)}$, then we let $X$ be a point in $ \mathcal{T}_{qc} (R)$
such that $d_{ls}(X,Y)< \frac{\epsilon}{2}$. If $Y\in \mathcal{T}_{qc} (R)$, we just set $X=Y$. By Proposition  \ref{pro:nodense}, there is a point $Z\in \mathcal{T}_{ls} (R) -  \overline{\mathcal{T}_{qc} (R)}$ such that $d_{ls}(X,Z)< \frac{\epsilon}{2}$. By the triangle inequality,
$d_{ls}(Y,Z)< \epsilon$.
\end{proof}

\section{Connectedness of the length-spectrum  Teichm\"uller space}    \label{sec:connected}

In this section, we will prove that, if $(R,\mathcal{P} )$ is upper-bounded but does not satisfy Shiga's property, then $(\mathcal{T}_{ls}(R),d_{ls} )$ is path-connected.

Assume that $\sup_{i}\ell_R(C_i)\leq M$. We associate to $R$ the Fenchel-Nielsen coordinates
$$\left( (\ell_R(C_i), \tau_R(C_i))  \right)_{i=1, 2, \cdots} ,$$ and we choose the twist coordinates in such a way that $|\tau_R(C_i)|< \ell_R(C_i)$.

Consider now an arbitrary $X\in \mathcal{T}_{ls}(R)$, with Fenchel-Nielsen coordinates
$$\left( (\ell_X(C_i), \tau_X(C_i))  \right)_{i=1, 2, \cdots} .$$ We write $X=(\ell_X(C_i), \tau_X(C_i))$ for simplicity.
\begin{lemma}\label{lem:bishop}
Suppose that $d_{ls}(R,X)< 2K$.  We can find a hyperbolic metric $Y\in \mathcal{T}_{qc}(R)$, with Fenchel-Nielsen coordinates $Y=(\ell_Y(C_i), \tau_Y(C_i))$, such that $\ell_Y(C_i)=\ell_X(C_i)$ and $|\tau_Y(C_i)-\tau_R(C_i)|<  2e^K \ell_R(C_i)\leq 2e^K M$.
\end{lemma}
\begin{proof}
Using a theorem of Bishop \cite{Bishop}, for any two geodesic pair of pants $P$ and $ Q$, there is a quasiconformal map from $P$ to $Q$, which satisfies
\begin{enumerate}
\item The quasiconformal dilatation of the map is less than $1+Cd_{ls}(P,Q)$, where $C>0$ depends on an upper bound of $d_{ls}(P,Q)$ and the boundary length of $P$.
\item The map  is affine on each of the boundary components.
\end{enumerate}

Cut $R$ into pairs of pants along $\mathcal{P}$. We use Bishop's construction \cite{Bishop} to deform each hyperbolic pair of pants, with boundary lengths $\ell_R(\cdot)$ into a pair of pants with boundary lengths $\ell_X(\cdot)$. Then we patch together the new hyperbolic pairs of pants to get a new hyperbolic metric $Y$ homeomorphic to $R$, in the following way.
Suppose two pairs of pants $P_1,P_2$ are  joined at $\alpha\in \mathcal{P}$ and deformed into $Q_1, Q_2$. Then we patch together $Q_1$ and $ Q_2$ by identifying the common image $\alpha$ by an affine map. In this way, the quasiconformal map between pairs of pants provided by Bishop can be extended to a global quasiconformal map between $R$ and $Y$, with dilatation bounded by
$$1+C\sup_{i} |\log \frac{\ell_X(C_i)}{\ell_R(C_i)}|,$$
where $C$ is a positive constant depending on $K, M$.

As a result,
$$2d_{qc}(R,Y)\leq \log (1+C\sup_{i} |\log \frac{\ell_X(C_i)}{\ell_R(C_i)}|)\leq C\sup_{i} |\log \frac{\ell_X(C_i)}{\ell_R(C_i)}|.$$

Let $Y=(\ell_Y(C_i), \tau_Y(C_i))$. Then, it follows from our construction that
$|\tau_Y(C_i)| \leq e^{K \ell_R(C_i)}$ and $|\tau_Y(C_i)-\tau_R(C_i)| \leq 2e^{K \ell_R(C_i)}$.
\end{proof}

Now we construct a continuous path in $\mathcal{T}_{ls}(R)$ from $Y$ to $X$ by varying the twist coordinates. First we need the following result.
\begin{theorem}\label{thm:iff}
The hyperbolic structure $X=(\ell_X(C_i), \tau_X(C_i))$ lies in $\mathcal{T}_{ls}(R)$ if and only if there is a constant $N>0$ such that for each $i$,
$$ \left|\log \frac{\ell_X(C_i)}{\ell_Y(C_i)}\right| < N $$
and
$$|\tau_X(C_i)-\tau_R(C_i)|< N \max \{|\log\ell_R(C_i)|, 1\}.$$

\end{theorem}
\begin{proof}
By Lemma \ref{lem:bishop}, we only need to consider the case where $\ell_X(C_i)=\ell_R(C_i)$ for each $i$. In this case, $X$ is a multi-twist deformation of $R$ along $\mathcal{P}$. If
for each $i$,
$$|\tau_X(C_i)-\tau_R(C_i)|< N \max \{|\log\ell_R(C_i)|, 1\},$$
then we can use the proof of Theorem  \ref{thm:main} to show that $d_{ls}(R, X)<\infty$.

Conversely, suppose that $d_{ls}(R, X)<\infty$ and that there is a subsequence $\{\alpha_n\}$  of $\mathcal{P}$, such that
$$|\tau_X(\alpha_n)-\tau_R(\alpha_n)|> N \max \{|\log\ell_R(\alpha_n)|, 1\}.$$

If there is a subsequence of $\{\alpha_n\}$, still denoted by $\{\alpha_n\}$, with length $\ell_R(\alpha_n)$ tending to zero, then, by Theorem \ref{thm:beta}, we have
$d_{ls}(R,X)=\infty$, which contradicts our assumption. If  $\{\ell_R(\alpha_n)\}$ is bounded below (and bounded above by assumption), then it also follows from Proposition 3.3 of
\cite{ALPS2} that $d_{ls}(R, X)=\infty$. As a result, there is a sufficiently large constant $N$ such that for each $i$,
$$|\tau_X(C_i)-\tau_R(C_i)|< N \max \{|\log\ell_R(C_i)|, 1\}.$$
\end{proof}
Now we can prove the following:
\begin{theorem} \label{thm:path}
If $(R,\mathcal{P} )$ is upper-bounded but does not satisfy Shiga's property, then
$(\mathcal{T}_{ls}(R),d_{ls} )$ is path-connected.

\end{theorem}
\begin{proof}
Given any $X=(\ell_X(C_i), \tau_X(C_i))$ in $\mathcal{T}_{ls}(R)$ with $d_{ls}(R,X)< 2K$, we first construct the point $Y=(\ell_Y(C_i), \tau_Y(C_i))$ as we did in Lemma \ref{lem:bishop}. Then we have $\ell_Y(C_i)=\ell_X(C_i)$ and $|\tau_Y(C_i)-\tau_R(C_i)|< 2KM$ for each $C_i$. By Theorem \ref{thm:iff}, we also have
$$|\tau_X(C_i)-\tau_Y(C_i)|< N \max \{|\log\ell_R(C_i)|, 1\},$$
if  we choose $N$ sufficiently large (depending on $K, M$).

Now we use the multi-twist deformation to construct a path in $\mathcal{T}_{ls}(R)$ connecting $Y$ and $X$, by letting
$Y_t=(\ell_X(C_i), (1-t)\tau_Y(C_i)+t\tau_X(C_i),  0 \leq t \leq 1$. Theorem \ref{thm:iff} shows that $Y_t$ lies in $\mathcal{T}_{ls}(R)$.

Moreover, we can use the proof of Proposition \ref{pro:boundary} to prove that
$$d_{ls}(Y_s,Y_t)\leq \frac{(N+2e^KM)|s-t|} {2C}.$$

If we take the Teichm\"uller geodesic path from $R$ to $Y$, and then take the path obtained by the multi-twist deformations from $Y$ to $X$, then we get a continuous (in fact, Lipschitz) path in $\mathcal{T}_{ls}(R)$ connecting $R$ and $X$.
\end{proof}

The following question remains open:
\begin{question}
Suppose that $(R,\mathcal{P} )$ is upper-bounded but does not satisfy Shiga's condition.  Is $(\mathcal{T}_{ls}(R),d_{ls} )$  contractible?
\end{question}

\medskip

 \noindent {\bf Acknowledgements.}
 L. Liu and W. Su were partially supported by NSFC; 
D. Alessandrini is supported by Schweizerischer Nationalfonds 200021\_131967/2.

\end{document}